\documentclass[12pt]{amsart}

\usepackage{amssymb,latexsym,amscd}

\usepackage{amsthm,amsmath}

\newcommand{\cO}{\mathcal{O}}

\newcommand{\cE}{\mathcal{E}}

\newcommand{\cF}{\mathcal{F}}

\newcommand{\cA}{\mathcal{A}}
\newcommand{\cB}{\mathcal{B}}

\newcommand{\DD}{\mathcal{D}}

\newcommand{\Sym}{\mathrm{Sym}}

\newcommand{\lra}{\longrightarrow}

\newcommand{\ra}{\rightarrow}

\newcommand{\PP}{\mathcal{P}}
\newcommand{\EE}{\mathcal{E}}

\newcommand{\Sbb}{\mathbb{S}}

\newcommand{\End}{\mathrm{End}}

\newcommand{\Hom}{\mathrm{Hom}}
\newcommand{\Ext}{\mathrm{Ext}}

\newcommand{\GL}{\mathrm{GL}}
\newcommand{\SL}{\mathrm{SL}}
\newcommand{\SO}{\mathrm{SO}}
\newcommand{\OO}{\mathrm{O}}
\newcommand{\Sp}{\mathrm{Sp}}

\newcommand{\im}{\mathrm{im}}

\newcommand{\coker}{\mathrm{coker}}

\theoremstyle{plain}

\newtheorem{thm}{Theorem}[section]

\newtheorem{lem}[thm]{Lemma}

\newtheorem{prop}[thm]{Proposition}

\begin{document}

\title[]{Orthogonal bundles over curves in characteristic two}

\begin{abstract}
Let $X$ be a smooth projective curve of genus $g \geq 2$ defined over a field of characteristic two. We give
examples of stable orthogonal bundles with unstable underlying vector bundles and use them
to give counterexamples to Behrend's conjecture on the canonical reduction of principal
$G$-bundles for $G = \SO(n)$ with $n \geq 7$.
\end{abstract}

\author{Christian Pauly}

\address{D\'epartement de Math\'ematiques \\ Universit\'e de Montpellier II - Case Courrier 051 \\ Place Eug\`ene Bataillon \\ 34095 Montpellier Cedex 5 \\ France}
\email{pauly@math.univ-montp2.fr}






\maketitle

Let $X$ be a smooth projective curve of genus $g \geq 2$ and let $G$ be a connected reductive 
linear algebraic group defined over a field $k$ of arbitrary characteristic.  One
associates to any principal $G$-bundle $E_G$ over $X$ a reduction $E_P$  of $E_G$ to
a parabolic subgroup $P \subset G$, the so-called {\em canonical reduction}
--- see e.g. \cite{Ra}, \cite{Be}, \cite{BH} or \cite{H} for its definition. We only mention
here that in the case $G = \GL(n)$ the canonical reduction coincides with the
Harder-Narasimhan filtration of the rank-$n$ vector bundle  associated to $E_G$.

\bigskip

In \cite{Be} (Conjecture 7.6) K. Behrend conjectured that for any principal $G$-bundle
$E_G$ over $X$ the canonical reduction $E_P$ has no infinitesimal deformations, or 
equivalently, that the vector space
$H^0(X, E_P \times^P \mathfrak{g}/ \mathfrak{p})$ is zero. Here $\mathfrak{p}$ and
$\mathfrak{g}$ are the Lie algebras of $P$ and $G$ respectively.  
Behrend's conjecture implies that the canonical reduction $E_P$ is defined over the
same base field as $E_G$. 

\bigskip

We note that this conjecture holds for the structure groups $\GL(n)$ and  $\Sp(2n)$ in any characteristic,
and also for $\SO(n)$ in any characteristic different from two --- see \cite{H} section 2.
On the other hand, a counterexample to Behrend's conjecture for the exceptional group $G_2$ in characteristic two has been
constructed recently by J. Heinloth in \cite{H} section 5.

\bigskip

In this note we focus on $\SO(n)$-bundles in characteristic two. As a starting point we consider
the rank-$2$ vector bundle $F_*L$ given by the direct image under the Frobenius map $F$ of a line bundle $L$ over 
the curve $X$ and observe (Proposition \ref{inclusionA}) that the $\SO(3)$-bundle $\cA := \End_0(F_*L)$ is stable, but that
its underlying vector bundle is unstable and, in particular, destabilized by the rank-$2$
vector bundle $F_*\cO_X$. We use this 
observation to show that the $\SO(7)$-bundle 
$$F_*\cO_X \oplus (F_* \cO_X)^* \oplus \cA$$ 
equipped
with the natural quadratic form gives a counterexample to Behrend's conjecture. Replacing
$\cA$ by $\hat{\cA} = \End(F_*L)$ we obtain in the same way a counterexample for $\SO(8)$, and more
generally for $\SO(n)$ with $n \geq 7$ after adding direct 
summands of hyperbolic planes.  Note that Behrend's conjecture holds for $\SO(n)$ with $n \leq 6$ because
of the exceptional isomorphisms with other classical groups.

\bigskip

The first three sections are quite elementary and recall well-known facts on quadratic forms,
orthogonal groups and their Lie algebras, as well as orthogonal bundles in characteristic two. In the last section 
we give an example (Proposition \ref{exso7}) of an unstable $\SO(7)$-bundle having its canonical reduction only defined
after an inseparable extension of the base field.

\bigskip

This note can be considered as an appendix to the
recent paper \cite{H} by J. Heinloth, whom I would like to thank for useful discussions.

\bigskip

\section{Quadratic forms in characteristic two}

The purpose of this section is to recall definitions and basic properties of quadratic forms in
characteristic two.

\bigskip

Let $k$ be an algebraically closed field of characteristic two and let $V$ be a vector space of
dimension $n$ over $k$. We denote by $\Sym^2 V$ and $\Lambda^2 V$ the symmetric and exterior square
of $V$ and by $F (V) : = V \otimes_k k$ the Frobenius-twist of $V$. We observe that there exists
a well-defined $k$-linear injective map
$$ F(V) \hookrightarrow \Sym^2 V \qquad v \otimes \lambda \mapsto \lambda ( v \cdot v),$$
where $v\cdot v$ denotes the symmetric square of $v$. More precisely, we have the following exact sequence
of $k$-vector spaces
$$ 0 \lra F(V) \longrightarrow \Sym^2 V \lra  \Lambda^2 V \lra 0.$$
Let $\sigma$ denote the involution on $V \otimes V$ defined by $\sigma (v \otimes w) = 
w \otimes v$ and let $\Sbb_2(V) \subset V \otimes V$ denote the $\sigma$-invariant subspace.
Then one has the exact sequence
$$ 0 \lra \Sbb_2(V) \lra V \otimes V \lra \Lambda^2 V \lra 0,$$
where the last arrow denotes the canonical projection of $V \otimes V$ onto $\Lambda^2 V$.
By choosing a basis of $V$, it can be checked that the image of $\Sbb_2(V)$ under the
canonical projection $V \otimes V \rightarrow \Sym^2 V$ equals the subspace
$F(V)$.

\bigskip

By definition a quadratic form on the vector space $V$ is an element $Q \in \Sym^2(V^*)$. Equivalently,
see e.g. \cite{Bo} section 23.5, a quadratic form can be seen as a $k$-valued function on $V$ such that
$$ Q(av + bw) = a^2 Q(v) + b^2 Q(w) + ab \beta(v,w) \qquad a,b \in k; v,w \in V, $$
where $\beta$ is a bilinear form on $V$. The bilinear form $\beta$ associated to a quadratic form $Q$ 
is determined by the formula
$$ \beta(v,w) = Q(v + w) + Q(v) + Q(w). $$
Is is clear that $\beta$ is symmetric, hence $\beta$ can be regarded as an element in
$\Sbb_2(V^*)$. The assignment $Q \mapsto \beta$ gives rise to the $k$-linear {\em polarisation map}
$$ \PP : \Sym^2(V^*) \longrightarrow \Sbb_2(V^*), \qquad Q \mapsto \beta. $$
Over a field $k$ of characteristic different from $2$ we recall that the polarisation map $\PP$ is an isomorphism. In our
situation one can easily work out that
$$ \ker \  \PP = F(V^*) \qquad \text{and} \qquad  \coker  \ \PP = F (V^*). $$
Quadratic forms $Q \in \ker \ \PP$ correspond to squares of linear forms on $V$ and the last equality asserts that
the bilinear form $\beta$ is alternating, i.e., $\beta(v,v) = 0$ for all $v \in V$.

\bigskip

We say that $Q$ is {\em non-degenerate} if $\beta$ induces an isomorphism $\tilde{\beta} : V \ra V^*$ for $n$
even, and if $\dim \ker \tilde{\beta} = 1$ and $Q_{| \ker \tilde{\beta}} \not= 0$ for $n$ odd. 
We say that a linear subspace $W \subset V$ is {\em isotropic} for $Q$ if $Q_{|W} = 0$.

\section{Orthogonal groups and their Lie algebras in characteristic two}

In this section we work out in detail the structure of the Lie algebras of the orthogonal groups
$\SO(7)$ and $\SO(8)$ over a field $k$ of characteristic two, as well as of some of their parabolic subgroups
$P \subset \SO(n)$. Following the principle of \cite{H}, we describe the representation 
$\mathfrak{so}(n)/\mathfrak{p}$ of the Levi subgroup $L \subset P$, which we will use later in section 5. The main reference is \cite{Bo} section 23.6. 

\bigskip

By definition the orthogonal group $\OO(n)$ is the subgroup of $\GL(n) = \GL(V)$ 
stabilizing a non-degenerate quadratic form $Q \in \Sym^2(V^*)$. Note that in characteristic two $\OO(n)
\subset \SL(n)$. We also recall that, if $n$ is odd, the group $\OO(n)$ is connected. If $n$ is even,
$\OO(n)$ has two connected components distinguished by the Dickson invariant (\cite{D} page 301). In
order to keep the same notation as in characteristic $\not= 2$, we denote by $\SO(n)$ the 
connected component of $\OO(n)$ containing the identity.

\subsection{The group $\SO(8)$}

We consider $V = k^8$ endowed with the non-degenerate quadratic form
$$ Q = x_1 x_7 + x_2 x_8 + x_3 x_5 + x_4 x_6 \in \Sym^2(V^*).$$
We choose this non-standard quadratic form in order to obtain a simple
description of the Lie algebra $\mathfrak{p}$ of the parabolic subgroup $P$ fixing an isotropic
$2$-plane, which will appear in section 5. The bilinear form $\beta$ associated to $Q$ is 
given by the matrix
$$ \left(
\begin{array}{cccc}
	 O & O & O & I \\
	 O & O & I & O \\
	 O & I & O & O \\
	 I & O & O & O \\
\end{array} \right),
$$
where $I$ denotes the identity matrix $\scriptsize{ \left( 
\begin{array}{cc}
	1 & 0  \\ 0 & 1 
\end{array}  \right)}$ and $O$ the zero $2 \times 2$ matrix. A straightforward computation shows that the Lie algebra $\mathfrak{so}(8)
\subset \End(V)$ of $\SO(8)$ consists of the matrices of the form
\begin{equation} \label{so8}
\left(
\begin{array}{cccc}
	 X_1 & X_2 & X_3 & D_1 \\
	 X_4 & X_5 & D_2 & {}^tX_3 \\
	 X_6 & D_3 & {}^tX_5 & {}^tX_2 \\
	 D_4 & {}^tX_6 & {}^tX_4 & {}^tX_1 \\
\end{array} \right),
\end{equation}
where the $X_i$ are $2 \times 2$ matrices, the ${}^tX_i$ denote their transpose, and 
the $D_i$ denote matrices of the form 
$\scriptsize{ \left( 
\begin{array}{cc}
	0 & \lambda  \\ \lambda & 0 
\end{array}  \right)}$.

We consider the maximal parabolic subgroup $P \subset \SO(8)$ preserving the isotropic $2$-plane
$W \subset V = k^8$ defined by the equations $x_i = 0$ for $i=3,\ldots, 8$. Its orthogonal space
$W^\perp$ is then defined by the equations $x_7 = x_8 = 0$. The Levi subgroup $L \subset P$ is 
isomorphic to $\GL(2) \times \SO(4)$. The Lie algebra $\mathfrak{p}$ of $P$
consists of the matrices \eqref{so8} that satisfy the conditions
$$ X_4 = X_6 = 0 \qquad \text{and} \qquad D_4 = 0. $$
The vector space $\mathfrak{so}(8)/ \mathfrak{p}$ sits naturally in the exact sequence
\begin{equation} \label{so8modp}
0 \lra \Hom(W, W^\perp/W) \lra \mathfrak{so}(8)/ \mathfrak{p} \lra \DD \lra 0, 
\end{equation}
where $\DD$ is the one-dimensional subspace generated by the element $e \otimes f + f \otimes e \in W^* \otimes W^* = 
\Hom(W,W^*) \cong \Hom(W, V/W^\perp)$, where $e$ and $f$ are linearly independent.

\subsection{The group $\SO(7)$}

We consider $V = k^7$ endowed with the non-degenerate quadratic form
$$ Q = x_1 x_6 + x_2 x_7 + x_3 x_4 + x_5^2 \in \Sym^2(V^*).$$
The bilinear form $\beta$ associated to $Q$ is given by the $7 \times 7$ matrix
$$ \left(
\begin{array}{cccc}
	 O & O & 0 & I \\
	 O & I & 0  & O \\
	 0 & 0 & 0 & 0 \\
	 I & O & 0  & O \\
\end{array} \right).
$$
The radical $\ker \tilde{\beta}$ of $Q$ is generated by $e_5$. A straightforward computation shows that the
Lie algebra $\mathfrak{so}(7) \subset \End(V)$ of $\SO(7)$ consists of the matrices of the form
\begin{equation} \label{so7}
\left(
\begin{array}{cccc}
	 X_1 & X_2 & 0 & D_1 \\
	 X_3 & \Delta & 0  & {}^tX_2 \\
	 x_1 & x_2 & 0 & x_3 \\
	 D_2 & {}^tX_3 & 0  & {}^tX_1 \\
\end{array} \right),
\end{equation}
where the $x_i$ are $2 \times 1$ matrices and $\Delta$ is a diagonal matrix of the form
$\scriptsize{ \left( 
\begin{array}{cc}
	\lambda & 0  \\ 0 & \lambda 
\end{array}  \right)}$.
We consider the maximal parabolic subgroup $P \subset \SO(7)$ preserving the isotropic $2$-plane
$W \subset V = k^7$ defined by the equations $x_i = 0$ for $i=3,\ldots, 7$. Its orthogonal space
$W^\perp$ is then defined by the equations $x_6 = x_7 = 0$. The Levi subgroup $L \subset P$ is
isomorphic to $\GL(2) \times \SO(3)$. The Lie algebra $\mathfrak{p}$ of $P$
consists of the matrices \eqref{so7} that satisfy the conditions
$$ X_3 = 0, \qquad x_1 = 0, \qquad \text{and} \qquad D_2 = 0.$$
As above we have an exact sequence
\begin{equation} \label{so7modp}
0 \lra \Hom(W, W^\perp/W) \lra \mathfrak{so}(7)/ \mathfrak{p} \lra \DD \lra 0. 
\end{equation}

\section{Orthogonal bundles}

By definition an {\em orthogonal bundle} over a smooth projective curve $X$ defined over an algebraically closed field
$k$ is a principal $\OO(n)$-bundle over $X$, which we denote by $E_{\OO(n)}$. Equivalently a $\OO(n)$-bundle
corresponds to a pair $(\EE,Q)$, where $\EE$ is the underlying rank-$n$ vector
bundle and $Q$ is a non-degenerate $\cO_X$-valued quadratic form on $\EE$, i.e., a non-degenerate
element $Q \in H^0(X, \Sym^2(\EE^*))$. By abuse of notation we will write $E_{\OO(n)} = 
(\EE,Q)$.

\bigskip

One observes (see e.g. \cite{Ram} Definition 4.1) that the notion of (semi-)stability of the $\OO(n)$-bundle
$E_{\OO(n)}$ translates into the following condition on the pair $(\EE,Q)$: we say that $(\EE,Q)$ is
semi-stable (resp. stable) as an orthogonal bundle if and only if 
$$ \mu (\cF) := \frac{\deg \cF}{\mathrm{rk} \cF} \leq 0 \qquad (\text{resp.} < 0) $$
for all isotropic subbundles $\cF \subset \EE$. 

\bigskip

A noteworthy result is the following

\begin{prop}[\cite{Ram} Proposition 4.2]
Assume that the characteristic of $k$ is different from two. The pair $(\EE,Q)$ is semi-stable as
an orthogonal bundle if and only if the underlying vector bundle $\EE$ is semi-stable.
\end{prop}

In the next section we will show that the assumption on the characteristic can not be removed.

\section{Stable orthogonal bundles with unstable underlying bundle}

We denote by $F : X \ra X$ the absolute Frobenius of the curve $X$ of genus $g \geq 2$ defined over an algebraically closed field $k$
of characteristic two. We denote by $K_X$ the canonical line bundle of $X$. 

\subsection{An orthogonal rank-$3$ bundle}

Let $E$ be a rank-$2$ vector bundle of arbitrary degree over $X$. The determinant endows the bundle $\End_0(E)$ 
of traceless endomorphisms of $E$ with an $\cO_X$-valued non-degenerate quadratic form
$$ \det : \End_0(E) \lra \cO_X.$$
Its associated bilinear form $\beta$ is given by $\beta(v,w) = \mathrm{Tr}(v \circ w)$ for $u,v$ local sections of
$\End_0(E)$ and the radical of $\det$ equals
the line subbundle $\cO_X \hookrightarrow \End_0(E)$ of homotheties.
\begin{lem}
There is an exact sequence
\begin{equation} \label{esend}
 0 \lra \cO_X \lra \End_0(E) \lra F^*(E) \otimes (\det E)^{-1} \lra 0, 
\end{equation} 
where the first homomorphism is the inclusion of homotheties into $\End_0(E)$. 
\end{lem}

\begin{proof}
Consider $V = k^2$ and $W = \End_0(V)$ equipped with the determinant. Note that for any $g \in \GL(V)$ the conjugation
$C_g : W \ra W$ is orthogonal and leaves the identity $\mathrm{Id}$ invariant. Thus we obtain group 
homomorphisms $\GL(V) \ra \SO(W) \ra \GL(W/ \langle \mathrm{Id} \rangle)$, $ g \mapsto C_g \mapsto \overline{C}_g$. A
straightforward computation shows that the composite map is given by $g \mapsto \frac{1}{\det g} F(g)$, where
$F(g)$ is the Frobenius-twist of $g$.
\end{proof}

\begin{prop} \label{stability}
If $E$ is stable, then $(\End_0(E), \det)$ is stable as an orthogonal bundle.
\end{prop}

\begin{proof}
It suffices to observe that an isotropic line subbundle $M$ of $\End_0(E)$ corresponds to
a nonzero homomorphism $\phi: E \ra E \otimes M^{-1}$ of rank $1$. If $L$ denotes
the line bundle $\im \phi \subset E \otimes M^{-1}$, we obtain the inequalities
$\mu(E) < \deg L < \mu( E \otimes M^{-1})$, which implies $\deg M < 0$.
\end{proof}

\bigskip

Let $B$ denote the sheaf of locally exact differentials \cite{R}, which can be defined as the cokernel 
\begin{equation} \label{extfrob}
0 \lra \cO_X \lra F_*\cO_X \lra B \lra 0,
\end{equation}
of the inclusion $\cO_X \hookrightarrow F_*\cO_X$ given by the Frobenius map. Note that in characteristic $2$ 
the sheaf $B$ is a theta-characteristic, i.e., $B^2 = K_X$. We denote by $e \in \Ext^1(B, \cO_X) = 
H^1(X, B^{-1})$  the non-zero extension class (unique up to a scalar) determined by the exact 
sequence \eqref{extfrob}. The next lemma immediately follows from the definition of $e$ (see also \cite{R}).

\begin{lem} \label{lemma1}
The extension class $e$ generates the one-dimensional kernel of the Frobenius map, i.e.,
$$ \langle e \rangle = \ker \left( F: H^1(X, B^{-1}) \longrightarrow H^1(X, K_X^{-1}) \right).$$
In particular, we have a direct sum
$$ F^* \left( F_* \cO_X \right) = \cO_X \oplus K_X. $$
\end{lem}

\bigskip

Let $L$ be a line bundle of arbitrary degree. Then the rank-$2$ vector bundle $F_* L$
has determinant equal to $B \otimes L$, is stable and is destabilized by pull-back
by the Frobenius map $F$ (\cite{LP}). More precisely, the bundle $F^*( F_* L) \otimes B^{-1} \otimes L^{-1}$ is the 
unique non-split extension  (see e.g. \cite{LS})
\begin{equation} \label{Gunningbundle}
 0 \longrightarrow B \stackrel{\iota}{\longrightarrow} F^* ( F_* L) \otimes B^{-1} \otimes L^{-1} \stackrel{\pi}{\longrightarrow} B^{-1}
\longrightarrow 0.
\end{equation}  

\begin{prop} \label{inclusionA}
\begin{itemize}
\item[(i)]
The vector bundle $\cA := \End_0(F_*L)$ does not depend on $L$. 
\item[(ii)]
The vector bundle $\cA$ is the unique non-split extension
$$ 0 \lra F_*\cO_X \stackrel{\phi}{\lra} \cA \stackrel{\psi}{\lra} B^{-1} \lra 0.$$
We denote by $\phi$ and $\psi$ generators of the one-dimensional spaces $\Hom(F_*\cO_X, \cA)$ and
$\Hom(\cA, B^{-1})$.
\item[(iii)]
The restriction of the quadratic form $\det$ to the subbundle $F_*\cO_X$ equals the 
evaluation morphism $F^* F_* \cO_X \longrightarrow \cO_X$. In particular the restriction of $\beta$
to $F_*\cO_X$ is identically zero.
\item[(iv)]
Let $x$ and $y$ be local sections of $F_*\cO_X$ and $\cA$. Then the two bilinear forms $\beta(\phi(x),y)$ and 
$\langle x, \psi(y) \rangle$ on the product $F_*\cO_X \times \cA$ differ by a non-zero multiplicative scalar. 
Here $\langle .,.\rangle$
denotes the standard pairing between $F_* \cO_X$ and its dual $(F_* \cO_X)^*$. 
\end{itemize}
\end{prop}

\begin{proof}
First we consider the exact sequence \eqref{esend} for the bundle $F_* L$
$$ 0 \lra \cO_X \lra \End_0(F_* L) \lra F^*(F_* L) \otimes B^{-1} \otimes L^{-1} \lra 0, $$
and call its extension class $f$. In order to show that there exists an 
inclusion $F_* \cO_X \hookrightarrow \End_0(F_* L)$, it is enough to show that $f$
pulls-back to $e$ under the inclusion $\iota : B \hookrightarrow F^* ( F_* L) \otimes B^{-1} \otimes L^{-1}$. 
Since by Lemma \ref{lemma1} the
extension class $e$ is characterized by the equality $F(e) = 0$ and since pull-back
commutes with Frobenius, it suffices to show that the extension 
$$F^* \End_0(F_* L) = \End_0(F^* F_* L) = \End_0(F^*( F_* L) \otimes B^{-1} \otimes L^{-1}) $$ 
contains the line subbundle $K_X = F^* B$ and that $\iota^*(f) \not= 0$. Since 
$F^* ( F_* L) \otimes B^{-1} \otimes L^{-1}$ is the non-split
extension \eqref{Gunningbundle}, we obtain a sheaf inclusion
\begin{equation} \label{nilpendo}
K_X = \Hom(B^{-1},B) \longrightarrow \End_0(F^* F_* L), \qquad \phi \mapsto \iota \circ \phi \circ \pi.
\end{equation}
Using a degree argument, one shows that there is an exact sequence 
\begin{equation} \label{es1}
0 \lra \cO_X \oplus K_X \lra \End_0(F^* F_* L) \lra \Hom( B, B^{-1}) = K_X^{-1} \lra 0.
\end{equation}
It remains to check that $\iota^*(f) \not= 0$. Suppose on  the contrary that $\iota^*(f) = 0$. Then
$B$ is a line subbundle of $\End_0(F_*L)$, hence we obtain a non-zero homomorphism
$F_*L \ra (F_*L) \otimes B^{-1}$. But $F_*L$ is stable and $\mu(F_*L) > \mu( F_*L \otimes B^{-1}) =
\mu (F_*L) +1 -g$, which implies $\Hom(F_*L,(F_*L) \otimes B^{-1}) = 0$, a contradiction.
Thus we have shown that $\End_0(F_*L)$ fits into the exact sequence
$$ 0 \lra F_*\cO_X \lra \End_0(F_*L) \lra B^{-1} \lra 0. $$
We denote by $a$ its extension class in $\mathrm{Ext}^1(B^{-1}, F_* \cO_X) = H^1(X, (F_* \cO_X) 
\otimes B)$. We have $(F_* \cO_X) \otimes B = (F_* \cO_X)^* \otimes K_X$ and by Serre duality
$H^1(X,(F_* \cO_X) \otimes B) = H^0(X, F_* \cO_X)^* = H^0(X, \cO_X)^*$, which implies that 
$\dim \mathrm{Ext}^1(B^{-1}, F_* \cO_X) = 1$. Hence there exists (up to isomorphism) a unique
non-split extension of $B^{-1}$ by $F_* \cO_X$. In order to show assertions (i) and (ii) it will
be enough to check that $a \not= 0$. But the push-out of $a$ under the map
$F_* \cO_X \ra B$ gives the non-split extension of $B^{-1}$ by $B$, which proves the claim. The 
proof of the equalities $\dim \Hom(F_*\cO_X, \cA) = \dim \Hom(\cA, B^{-1}) = 1$ is standard. Note that
\eqref{es1} is the pull-back under $F$ of the exact sequence in (ii).

\bigskip

The restriction of the bilinear form $\beta$ to $F_* \cO_X$ is identically zero, since $\beta$
factorizes through the line bundle quotient $B$ and $\beta$ is alternating (see section 1). Hence the
restriction of the quadratic form $\det$ lies in $F^* (F_* \cO_X)^* \subset \Sym^2 (F_* \cO_X)^*$, i.e.,
corresponds to an $\cO_X$-linear map $\alpha : F^*F_* \cO_X = \cO_X \oplus K_X  \ra \cO_X$. We note that
$\alpha_{|\cO_X} = \mathrm{id}$ and $\alpha_{|K_X} = 0$, since the summand $K_X$ corresponds to nilpotent
endomorphisms --- see \eqref{nilpendo}. This proves (iii).

\bigskip

Finally we observe that the bilinear form $(x,y) \mapsto \beta(\phi(x),y)$ factorizes through
the quotient $B \times B^{-1}$, since $\beta$ is identically zero on $F_*\cO_X$ by part (iii) and
since the radical of $\beta$ equals $\cO_X$. On the other hand it is easily seen that the bilinear
form $(x,y) \mapsto \langle x, \psi(y) \rangle$ also factorizes through $B \times B^{-1}$,
which proves (iv).
\end{proof}

\bigskip

Since $F_* \cO_X$ is stable and $\mu(F_* \cO_X) = \frac{g-1}{2} > 0$, we deduce that the Harder-Narasimhan
filtration of $\cA$ is
$$ 0 \subset F_* \cO_X \subset \cA. $$
However, since $F_*L$ is stable, Proposition \ref{stability} implies that the 
orthogonal bundle $(\cA, \det)$ is stable.

\bigskip
\noindent
{\bf Remark.} Similarly, it can be shown that the rank-$4$ vector bundle $\hat{\cA} := \End(F_*L)$ does
not depend on the line bundle $L$, that it fits 
into the exact sequence 
$$ 0 \lra F_* \cO_X \lra \End(F_*L) \lra \left( F_* \cO_X \right)^* \lra 0,$$
giving its Harder-Narasimhan filtration, and that $(\hat{\cA}, \det)$ is a stable orthogonal bundle.

\subsection{An orthogonal rank-$4$ bundle}

Besides the example given in the last remark of section 3.1, we can consider the unstable
decomposable rank-$4$  vector bundle
$$ \cB = F_*\cO_X \oplus (F_*\cO_X)^*$$
and endow it with the $\cO_X$-valued quadratic form
$$ Q(x +x^*) = q(x) + q^*(x^*) + \langle x,x^* \rangle, $$
where $x,x^*$ are local sections of the rank-$2$ bundles $F_*\cO_X$ and $(F_*\cO_X)^*$
respectively. The bracket $\langle x,x^* \rangle$ denotes the standard pairing between 
$F_*\cO_X$ and its dual $(F_*\cO_X)^*$ and $q,q^*$ are projections onto $\cO_X$ of the
direct sums (see Lemma \ref{lemma1})
\begin{equation} \label{defquadq}
 F^* \left( F_* \cO_X \right) = \cO_X \oplus K_X, \qquad F^* \left( F_* \cO_X \right)^* = \cO_X \oplus 
K_X^{-1}. 
\end{equation}
It is clear that $Q$ is non-degenerate and that $Q$ restricts to $q$ and $q^*$ on the direct
summands $F_*\cO_X$ and  $(F_*\cO_X)^*$ respectively.

\begin{prop}
The orthogonal rank-$4$ bundle $(\cB,Q)$ is stable.
\end{prop}

\begin{proof}
Let $M$ be an isotropic line subbundle of $F_*\cO_X \oplus (F_*\cO_X)^*$ and assume
that $\deg M \geq 0$. By stability of $(F_*\cO_X)^*$ the line bundle $M$ is 
contained in $F_*\cO_X$, which gives by adjunction a nonzero map $F^*M \ra
F^*F_*\cO_X \ra \cO_X$, contradicting isotropy of $M$. Next, let $S$ be an isotropic rank-$2$
subbundle of $F_*\cO_X \oplus (F_*\cO_X)^*$ and assume $\mu(S) \geq 0$. By the previous
considerations $S$ is stable, hence there is no non-zero map $S \ra (F_*\cO_X)^*$. Therefore
$S$ is a rank-$2$ subsheaf of $F_*\cO_X$, but this contradicts isotropy of $S$.
\end{proof}

\noindent
{\bf Remark.} Clearly the two bundles $\hat{\cA}$ and $\cB$ are non-isomorphic.

\section{Counterexamples to Behrend's conjecture}

We consider the rank-$7$ vector bundle
$$ \cE =  F_*\cO_X \oplus \cA \oplus (F_*\cO_X)^*$$
equipped with the non-degenerate quadratic form $Q$ defined by
$$Q(x + y + x^* ) = \langle x, x^* \rangle + \det y ,$$
where $x, x^*$ and $y$ are local sections of $F_* \cO_X$, $(F_* \cO_X)^*$ and
$\cA$ respectively. Note that the $\SO(7)$-bundle  $E_{SO(7)} = (\cE,Q)$ has a reduction to
the  Levi  subgroup 
$$L = \GL(2) \times \SO(3) \subset P \subset \SO(7)$$ 
and that its $L$-bundle $E_L$ given by the 
pair  $(F_*\cO_X, \cA)$ is stable, since $F_* \cO_X$ and $\cA$ are stable $\GL(2)-$ and $\SO(3)$-bundles respectively. 
Since $\mu(F_* \cO_X) > \mu(\cA) = 0$, we obtain that the canonical reduction
of the unstable $\SO(7)$-bundle $(\cE,Q)$ is given by the $P$-bundle $E_P := E_L\times^L P$.
Using the exact sequence \eqref{so7modp} we obtain the equality
$$ \Hom(F_*\cO_X, \cA) = H^0(X, E_P \times^P \mathfrak{so}(7)/\mathfrak{p} ).$$
Note that $H^0(X, E_P \times^P \DD) \subset \Hom(F_* \cO_X, (F_* \cO_X)^*) = 0$ by
stability of $F_* \cO_X$. But the former space is nonzero by Proposition \ref{inclusionA} (ii).
Hence this provides a counterexample to Behrend's conjecture.

\bigskip

The following unstable $\SO(8)$-bundles also provide counterexamples
to Behrend's conjecture
$$ F_*\cO_X \oplus \hat{\cA} \oplus (F_*\cO_X)^* \qquad \text{and} \qquad
F_*\cO_X \oplus \cB \oplus  (F_*\cO_X)^*,$$
with the quadratic form given by the standard hyperbolic form on the two
summands $F_*\cO_X \oplus (F_*\cO_X)^*$ and the quadratic forms on  $\hat{\cA}$
and $\cB$ introduced in sections 4.1 and 4.2. Similarly we use the exact sequence \eqref{so8modp}
and Proposition \ref{inclusionA} (ii) to show that the corresponding spaces
$H^0(X, E_P \times^P \mathfrak{so}(8)/\mathfrak{p} )$ are non-zero.

\bigskip

\noindent
{\bf Remark.} One can work out the relationship between the unstable $\SO(7)$-bundle $(\cE,Q)$ and the 
unstable $G_2$-bundle $E_{G_2}$ constructed in \cite{H} section 5: first, we recall that $G_2 \subset \GL(6)$. Consider the
symplectic rank-$6$ bundle $\overline{\cE}$ obtained from $(\cE,Q)$ under the purely inseparable group
homomorphism $\SO(7) \ra \Sp(6) \subset \GL(6)$. Then the bundle $\overline{\cE}$ has a reduction to
$E_{G_2}$.

\section{Non-rationality of the canonical reduction of an $\SO(7)$-bundle}

This section is largely inspired by the last remark of \cite{H}. We consider the unstable $\SO(7)$-bundle
$(\EE, Q)$ introduced in section 5. We denote by $K$ the field $k(t)$, by $X_K$ the curve  $X \times_k K$ and by
$\EE_K$ the vector bundle over $X_K$ obtained as pull-back of $\EE$ under the field extension $K/k$. We will
consider the quadratic form $\widetilde{Q}_K$ on the bundle $\EE_K$ defined by 
$$ \widetilde{Q}_K (x+y+x^*) = \langle x,x^* \rangle + \det y + t q(x),$$
where $x,x^*$ and $y$ are local sections of $F_*\cO_X$, $(F_*\cO_X)^*$ and $\cA$ respectively and 
$q$ is the quadratic form on $F_* \cO_X$ defined in \eqref{defquadq}. Note that by
Proposition \ref{inclusionA} (iii) one has $q(x) = \det (\phi(x))$.

\begin{prop} \label{exso7}
\begin{itemize}
\item[(i)] The quadratic form $\widetilde{Q}_K$ is non-degenerate.
\item[(ii)] The canonical reduction of the $\SO(7)$-bundle $(\EE_K, \widetilde{Q}_K)$ over $X_K$ is not
defined over $K$, but only over the inseparable quadratic extension $K' = K[s]/(s^2 -t)$.

\end{itemize}

\end{prop}

\begin{proof}
We consider the vector bundle $\EE_{K'}$ over $X_{K'}$ obtained from $\EE_K$ under the field extension $K'/K$
and introduce the automorphism $g_s$ of $\EE_{K'}$ defined by the matrix
$$
\left(
\begin{array}{ccc} 
\mathrm{Id}_{F_*\cO_X} & 0 & 0 \\
s \phi & \mathrm{Id}_{\cA} & 0 \\
0 & s \psi & \mathrm{Id}_{(F_*\cO_X)^*}
\end{array}
\right).
$$
We choose the generators $\phi \in \Hom(F_*\cO_X, \cA)$ and $\psi \in \Hom(\cA, B^{-1})$ such that
$\beta(\phi(x) , y ) = \langle x, \psi (y) \rangle $ for any local sections $x,y$ of $F_* \cO_X$ and
$\cA$ respectively --- see  Proposition \eqref{inclusionA} (iv). Note that $g_s$ is an involution.
Then we have
$$ g_s (x,y,x^*) = (x,y+s\phi(x), x^* + s \psi(y)).$$
We also denote by $Q$  and  $\widetilde{Q}_K$ the quadratic forms on $\EE_{K'}$ obtained from $Q$  and 
$\widetilde{Q}_K$ under the field
extensions $K'/k$ and $K'/K$ respectively. Then we have the equality
\begin{equation} \label{quadf}
Q(g_s (x,y,x^*)) = Q(x,y,x^*) + s^2 q(x) =  \widetilde{Q}_K(x,y,x^*),
\end{equation}
i.e., the quadratic forms $Q$ and $\widetilde{Q}_K$ on $\EE_{K'}$ differ by the automorphism $g_s$. In 
particular $\widetilde{Q}_K$ is non-degenerate and the pair $(\EE_K, \widetilde{Q}_K)$ determines
an $\SO(7)$-bundle over the curve $X_K$.

\bigskip

Because of \eqref{quadf} the canonical reduction of the $\SO(7)$-bundle $(\EE_K, \widetilde{Q}_K)$
is given by the isotropic rank-$2$ vector bundle $\alpha : F_* \cO_X \hookrightarrow \EE_{K'}$ with
$\alpha(x) = (x, s \phi(x), 0)$ for any local section $x$ of $F_* \cO_X$. It is clear that the
inclusion $\alpha$ is not defined over $K$.
\end{proof}


\end{document}